\documentclass[11pt]{amsart}
\usepackage{comment}

\usepackage{geometry}                
\usepackage{graphicx}  
\usepackage{subfig}
\usepackage{amssymb}
\usepackage{color}
\newtheorem{theorem}{Theorem}[section]    
\newtheorem{lemma}[theorem]{Lemma}          
\newtheorem{proposition}[theorem]{Proposition}  
\newtheorem{claim}[theorem]{Claim}  

\newtheorem{corollary}[theorem]{Corollary} 

\theoremstyle{definition}
\newtheorem{definition}[theorem]{Definition}

\newtheorem{remark-number}[theorem]{Remark}  
\newtheorem{question}[theorem]{Question} 
 
\newtheorem{example}[theorem]{Example}   
\newtheorem{remark}{Remark}             
\numberwithin{equation}{section}

\newcommand{\Int}{\textrm{Int}}

\newcommand{\Z}{\mathbb{Z}}
\newcommand{\N}{\mathbb{N}}

\newcommand{\im}{{\tt Im}}
\newcommand{\QP}{{\it QP}}
\newcommand{\Dehn}{{\it Dehn}}
\newcommand{\SMod}{{\rm SMod}}
\newcommand{\Mod}{{\rm Mod}}
\newcommand{\Homeo}{{\rm Homeo}^+}
\newcommand{\e}{\varepsilon}

\title{Positive factorizations of symmetric mapping classes}

\author{Tetsuya Ito}
\address{Department of Mathematics, Graduate School of Science, Osaka University \\ 1-1 Machikaneyama Toyonaka, Osaka 560-0043, JAPAN}
\email{tetito@math.sci.osaka-u.ac.jp}
\urladdr{http://www.math.sci.osaka-u.ac.jp/~tetito/}
\author{Keiko Kawamuro}
\address{Department of Mathematics,   
The University of Iowa, Iowa City, IA 52242, USA}
\email{keiko-kawamuro@uiowa.edu}
\date{\today}

\begin{document}

\begin{abstract}
We study Question 7.9 in \cite{EV} raised by Etnyre and Van Horn-Morris; whether a symmetric mapping class admitting a positive factorization is a lift of a quasi-positive braid. We answer the question affirmatively for mapping classes satisfying certain cyclic conditions. 
\end{abstract}

\maketitle

\section{Introduction}\label{section:intro}

In \cite[Question 7.9]{EV} Etnyre and Van Horn-Morris ask the following question:
\begin{question}\cite[Question 7.9]{EV}
\label{question:main}
If a symmetric mapping class $\phi\in\Mod(S)$ admits a positive factorization, then is $\phi$ a lift of a quasi-positive braid ?
\end{question}

This is a profound question connecting three important objects in topology; (1) symmetric mapping classes, (2) positive factorizations, and (3) quasipositive brads. 
We describe each here. 

Let $S$ be a compact, oriented surface with non-empty boundary, which is a special (see Section~\ref{section:BHtheory}) cyclic branched covering of a disk $D$ branched at $n$-points. 
A mapping class $\phi \in \Mod(S)$ is called \emph{symmetric} (fiber-preserving) \cite[p.65]{B} if $\phi$ is a lift of an element of the braid group $B_{n}=\Mod(D^{2} \setminus \{n \mbox{ points}\})$. 
Symmetric mapping class groups were introduced and studied by Birman and Hilden in a series of papers culminating in \cite{BH}.  
As Margalit and Winarski say in \cite{MW}, the Birman-Hilden theory has had influence on many areas of mathematics, from low-dimensional topology, to geometric group theory, to representation theory, to algebraic geometry and more.

A \emph{positive factorization} of a mapping class $\phi \in \Mod(S)$ is a factorization of $\phi$ into positive (right-handed) Dehn twists about simple closed curves. 
In contact and symplectic geometry, positive factorizations of  mapping classes play an important role due to the following fact: 
A contact $3$-manifold is Stein fillable if and only if it is supported by an open book whose monodromy admits a positive factorization \cite{AO, GI, LP}. 

A {\em quasipositive} braid in $B_n$ is a braid which factorizes into positive half twists about proper simple arcs in the $n$-punctured disk. 
Quasipositive knots and links are introduced and studied by Rudolph in a series of papers. 
Rudolph showed \cite{R} that a quasipositive knot can be realized as an intersection (transverse $\mathbb C$-link) of the unit sphere in $\mathbb C^2$ with an algebraic complex curve in $\mathbb C^2$. Boileau and Orevkov \cite{BO} proved the converse that any knot arising as such an intersection must be quasipositive.

In this paper, we give partial answers to the question of Etnyre and Van Horn-Morris.

Let $S$ be a compact, oriented surface with non-empty boundary. 
Let $D$ be a disk.  
Suppose $\pi: S \to D$ is a special $k$-fold cyclic branched covering of the disk $D$ branched at $n$ points. The meaning of ``special'' will be made clear in Section~\ref{section:BHtheory}. 
In \cite{BH} Birman and Hilden show that there is a well-defined injective homomorphism $\Psi: B_n \to \Mod(S)$ whose image is the symmetric mapping class group $\SMod(S)$ which is defined in Section~\ref{section:BHtheory}.


Using the homomorphism $\Psi:B_n \to \SMod(S)$, Etnyre and Van Horn-Morris in \cite{EV} consider various submonoids in $B_n$. 
\begin{eqnarray*}
P(n) 
&:=& 
\left\{ b \in B_n \ | \ b \mbox{ is a positive braid } \right\} \\ 
QP(n) 
&:=& 
\left\{ b \in B_n \ | \ b \mbox{ is a quasi-positive braid } \right\} \\
Dehn^+(n, k) 
&:=& 
\Psi^{-1}(Dehn^+(S)) \\
Tight^+(n, k) 
&:=& 
\Psi^{-1}(Tight^+(S)) \\
RV(n) 
&:=& 
\left\{ b \in B_n \ | \ b \mbox{ is a right-veering braid }
\right\} \\
Veer^+(n, k) 
&:=& 
\Psi^{-1}(Veer^+(S)) 
\end{eqnarray*}
Here, a braid $b \in B_{n}$ is \emph{positive} if it is a product of standard generators $\sigma_{1},\ldots,\sigma_{n-1}$, and is \emph{quasi-positive} if it is a product of conjugates of $\sigma_1$. 
We have $P(2)=\QP(2)$ and $P(n) \subsetneq QP(n)$ for $n>2$. 
The set $Dehn^{+}(S) \subset \Mod(S)$ is a monoid generated by positive Dehn twists, $Tight^{+}(S)\subset \Mod(S)$ is a monoid consisting of  monodromies supporting tight contact structures, and $Veer^{+}(S)\subset \Mod(S)$ is a monoid consisting of right-veering mapping classes.
One can see that $\Psi(b)$ is right-veering if and only if $b$ is right-veering (see \cite[Section 7]{IK2} for the definition(s) of right-veering braids).

\begin{proposition}
We have $Veer^+(n, k) = RV(n)$ for all $n$ and $k$. 
\end{proposition}

Etnyre and Van Horn-Morris observe that \cite[Lemma 3.1]{HKP} implies the following:  
\begin{proposition}\label{subset} \cite[p.355]{EV}
For all $n \geq 2$ and $k \geq 2$ we have $\QP(n) \subset \Dehn^+(n, k)$. 
\end{proposition}

We prove Proposition~\ref{subset} in Section~\ref{section:BHtheory}.

In summary, we have;
$$
P(n) \subset QP(n) \subset Dehn^+(n, k) \subset Tight(n,k) \subset Veer^+(n, k) = RV(n) \subsetneq B_n. 
$$

In \cite[Example 2.9]{HKM}, the strictness of the inclusion $QP(3)\subset Veer^+(3, 2)$ is shown. 

\begin{proposition}
In general, both the inclusions $Dehn^+(n, k) \subset Tight(n,k) \subset Veer^+(n, k)$ are strict as we show below: 
\end{proposition}

\begin{proof}
Let $\beta = (\sigma_{1}\sigma_{2})^{6} \sigma_{1}^{-13} \in B_{3}$. 
By \cite[Theorem 1.2]{HKM} 
the braid $\beta$ is in $Tight(3, 2)$ since the fractional Dehn twist coefficient of $\Psi(\beta) \in \Mod(S)$ is one. 
However, $\beta$ is not in $Dehn^{+}(3,2)$ since its exponent sum is negative which means $b \not\in QP(3)=Dehn^+(3, 2)$, where the equality $QP(3)=Dehn^+(3, 2)$ will be proved in Theorem~\ref{theorem:main2}. (See \cite[Corollary 3.6]{HKM} for a better criterion for $b \not \in QP(n)$ than a naive exponent sum argument.)

To see $Tight(n,k) \subsetneq Veer^+(n, k)$, recall \cite[Proposition 3.1]{PL} (for the open book $(D^2, id)$) and \cite[Proposition 3.14]{IK2} (for general open books) that every braid with respect to an open book is transversely isotopic to a right-veering braid after suitable positive stabilizations.
Let $\beta \in B_{n}$ be a right-veering braid that is a stabilization of $\sigma_1^{-1} \in B_2$. Then for each $k\geq 2$, $\beta$ is in $Veer^{+}(n,k)$ but not in $Tight(n,k)$.
\end{proof}

With these terminologies, Question 1.1 of Etnyre and Van Horn-Morris is equivalent to the following:

\begin{question}\cite[Question 7.9]{EV}
Do we have $\QP(n) = \Dehn^+(n, k)$? 
\end{question}

In \cite{EV} they say ``the answer is almost certainly no''. 
Thus our goal can be set to find sufficient conditions for $\QP(n) \supset \Dehn^+(n, k)$.

\subsection{Motivation}
Our particular branched covering $\pi:S\to D$ is closely related to the cyclic branched covering of the standard contact 3-sphere $(S^{3},\xi_{std}=\ker\alpha_{std})$. 
Let $K=\widehat{b}$ be a transverse knot in $(S^{3},\xi_{std})$ represented by the closure $\widehat b$ of an $n$-braid $b\in B_n$ with respect to the open book $(D^{2},id)$. Let $p: M_{K, k} \rightarrow S^{3}$ be the $k$-fold cyclic branched covering, branched along $K$. Then $M_{K,k}$ is equipped with a contact structure $\xi_{K,k}$ that is 
a perturbation of the kernel of the pull-back $p^{*}(\alpha_{std})$. Such a contact structure is supported by the open book $(S,\Psi(b))$. Thus, $(M_{K,k}, \xi_{K,k}) \simeq (M_{(S,\Psi(b))}, \xi_{(S,\Psi(b))})$.

Let $B^{4} (\subset {\mathbb C}^4)$ be the unit complex ball giving a Stein filling of $(S^{3},\xi_{std})$.
If the braid $b \in B_n$ is quasi-positive, a factorization of $b$ as a product of positive half twists gives rise to an immersed Seifert surface of $K=\widehat b$ with ribbon intersections as shown in \cite[Figure 9]{EV}. 
Pushing this surface into the interior of $B^{4}$ we have a properly  embedded symplectic surface $\Sigma$ in $B^{4}$ such that $\Sigma \cap \partial B^{4} = \partial \Sigma = K$. 
Let $W$ be the $k$-fold cyclic branched cover of $B^{4}$ branched along $\Sigma$. Then $W$ gives a Stein filling of $(M_{K,k},\xi_{K,k})$ 

On the other hand, a factorization of $b\in B_n$ into positive half twists induces a factorization of $\Psi(b) \in \Mod(S)$ into positive Dehn twists  (Proposition~\ref{subset}), to which one can associate a Legendrian surgery diagram (see Section~\ref{section:BHtheory} and \cite[Fig. 12]{HKP}). 
Let $X$ be the 4-dimensional handlebody obtained by attaching 2-handles to $B^4$ according to the Legendrian surgery diagram. 
By \cite{Eli}, \cite[Proposition 2.3]{G} the manifold $X$ is 
a Stein filling of $(M_{(S,\Psi(b))}, \xi_{(S,\Psi(b))})$. 

In fact, these two constructions of Stein fillings give the same manifold. Namely, the two 4-manifolds $X$ and $W$ are diffeomorphic, which follows from the proof of  \cite[Claim 2.1]{IT}. 
Thus, the branched covering construction behaves nicely not only for contact structures but also for Stein fillings.

With the above discussion in mind, we may extend Question \ref{question:main} to the following question about Stein filling and $\Z \slash k\Z$-action:

Assume that $b\in \Dehn^+(n, k)$. 
Then $(M_{K,k},\xi_{K,k})\simeq (M_{(S,\Psi(b))}, \xi_{(S,\Psi(b))})$ is Stein fillable because 
$\Psi(b) \in \Dehn^+(S)$ and $\Dehn^+(S)$ is contained in the monoid  $Stein(S)$ of Stein fillable open books \cite{Eli, GI}. 
The contact manifold $(M_{K,k},\xi_{K,k})$ also admits an $\Z\slash k\Z$-action as a contactomorphism with the quotient space $(S^{3},\xi_{std})$. Our question is: 

{\em Can we find a Stein filling $X$ coming from the above construction; namely, can the $\Z \slash k\Z$-action on $(M_{K,k},\xi_{K,k})$ extend to a holomorphic $\Z \slash k\Z$-action on some $X$ with the quotient space $B^4$? }

This new question suggests that, even if Question \ref{question:main} may have negative answer in general as Etnyre and Van Horn-Morris expect in \cite{EV}, it is worth trying to find sufficient conditions for $\QP(n) \supset \Dehn^+(n, k)$ to be hold.

\subsection{Main results} 
The following are our main results that give sufficient conditions for $\QP(n) \supset \Dehn^+(n, k)$. 

\noindent
{\bf Theorem~\ref{theorem:main2}}. {\em
For $n\leq 4$, $\QP(n) = \Dehn^+(n,2)$. \\
}

\noindent
{\bf Theorem~\ref{theorem:twist}}. {\em 
Let $\iota:S\to S$ be a deck transformation. 
Let $C$ be a simple closed geodesic curve in $S$ such that $C, \iota(C), \dots, \iota^{e-1}(C)$ are pairwise disjoint and $\iota^e(C)=C$ for some $e\in\{1,\dots,k\}$ that divides $k$. 
Let $d, j \in \N$. 
Suppose that $b^d \in \Dehn^+(n, k)$
with  
$$
\Psi(b^d)= \left(T_C \circ T_{\iota(C)} \circ T_{\iota^2(C)} \circ\dots\circ T_{\iota^{e-1}(C)}\right)^j.$$ 
Then $b \in QP(n)$ $($and so $b^d \in QP(n)).$ \\
}

This theorem states that Question \ref{question:main} has an affirmative answer for a symmetric mapping class which is a root of symmetric product of positive Dehn twists (see also Corollary \ref{cor:froot}). \\

\noindent
{\bf Theorem \ref{thm:general}.} {\em
Suppose that the subsurfaces $S'_1,\ldots, S'_k \subset S$ are pairwise non-isotopic. 
Assume that $[f]\in Dehn^{+}(S')$ is either 
\begin{itemize}
\item a non-negative power of a single Dehn twist (i.e., $S'$ is an annulus which is neighborhood of a simple closed geodesic curve), or 
\item a pseudo-Anosov map (i.e., $S'$ is a hyperbolic surface).
\end{itemize} 
Suppose that $b \in Dehn^+(n,k)$ satisfies 
\[ \Psi(b)=[f_1 f_2 \cdots f_{k}] \]
then $b \in QP(n)$.
}

We say that a simple closed curve $C$ is \emph{symmetric} if it is invariant under some deck translation. Question \ref{question:main} can be understood as a question 
of the existence of factorizations of elements of $Dehn^+(S) \cap \SMod(S)$ into positive Dehn twists about symmetric simple closed curves. 


A well-known example where a positive factorization coincides with a product of Dehn twists about symmetric simple closed curves may be the daisy relation \cite{EMV}.  
In Example \ref{example-daisy} we see that the technical looking assumptions in Theorem \ref{thm:general} can be understood as a generalization of the setting of the daisy relation, and view Theorem \ref{thm:general} as a generalization of the daisy relation.

\subsection{Organization of the paper} 

In Section \ref{section:BHtheory} we review results of Birman and Hilden that we need in this paper. 

In Section \ref{section:panswer}, we show that the answer to the question is affirmative when the number of branch points $n$ is two (Theorem \ref{theorem:main1}) or the degree of the branched covering is two with $n\leq 4$ branch points (Theorem \ref{theorem:main2}). 

In Section \ref{section:QP}, we discuss roots of quasi-positive braids and find conditions that a root of a quasi-positive braid is also quasi-positive. 
We obtain results that are used in Section~\ref{section:suff}.

In Section \ref{section:suff}, 
we prove our main results (Theorems~\ref{theorem:twist} and  \ref{thm:general}).

\section{Birman-Hilden theory} 
\label{section:BHtheory}

Throughout the paper, unless otherwise stated, we always assume that the boundary of a surface is non-empty, any homeomorphism of a surface with marked points (punctures) fixes the boundary pointwise and permutes the marked points, and any isotopy of homeomorphisms $\{f_{t}\}$ pointwise fixes  the boundary and the marked points.
We denote by $[f] \in\Mod(S)$ the isotopy class of a homeomorphism $f \in \Homeo(S)$. 
For a simple closed curve $C$ in a surface $S$, we denote by $t_{C} \in \Homeo(S)$ a right-handed Dehn twist about $C$ which is  supported in a neighborhood of $C$, and denote its isotopy class by $T_{C}\in\Mod(S)$.
A simple closed curve in a surface is called \emph{essential} if it is not homotopic to a point, a puncture, or a boundary component of the surface.

Let $P =\{p_1,\ldots,p_n \}\subset \Int (D)$ be $n$ points in the interior of a 2-disk $D$. We may identify the braid group $B_{n}$ with the mapping class group $\Mod(D\setminus P)$ of the $n$-punctured disk $D \setminus P$. 
The fundamental group $\pi_{1}(D\setminus P)$  is the free group of rank $n$ and generated by $\{x_1,\ldots,x_n\}$, where $x_{i}$ is a loop winding once around $p_i$ counter-clockwise. 
For $k\in \N$ let $e_k: \pi_{1}(D\setminus P) \rightarrow \Z\slash k\Z$
be a homomorphism defined by $e_{k}(x_i) = 1$ for all $i=1,\ldots,n$.
For $n\geq 2$ and $k\geq 2$, let $$\pi: S \approx S_{n,k} \to D$$ be the $k$-fold cyclic branched covering branched at $P$ such that the associated regular covering $S \setminus \pi^{-1}(P) \rightarrow D \setminus P$ is the $k$-fold cyclic cover corresponding to the subgroup $\ker(e_{k})$ of $\pi_1(D \setminus P)$. We denote by $\widetilde P =\{\widetilde{p_1},\ldots,\widetilde{p_{n}}\}= \pi^{-1}(P) \subset S$ the set of branch points in $S$. 
Let $\iota= \iota_{k} : S \to S$ be a generator of the deck transformation group ${\rm Aut}(S, \pi) \simeq \Z/k\Z$. 

We visualize $S$ and $\iota$ as follows. The covering $S$ can be viewed as the canonical Seifert surface of the $(n,k)$-torus link represented as the closure of the $n$-braid $(\sigma_{1}\cdots \sigma_{n-1})^{k}$ 
(see Figure~\ref{torus-link}). The deck transformation $\iota_k$ is the $\frac{2\pi}{k}$ rotation of the surface $S$ about the braid axis through the branch points $\widetilde P$, 
and  $\partial S$ is not pointwise fixed by $\iota_k$.  

\begin{figure}[htbp]
\includegraphics*[bb=71 564 384 745,width=110mm]{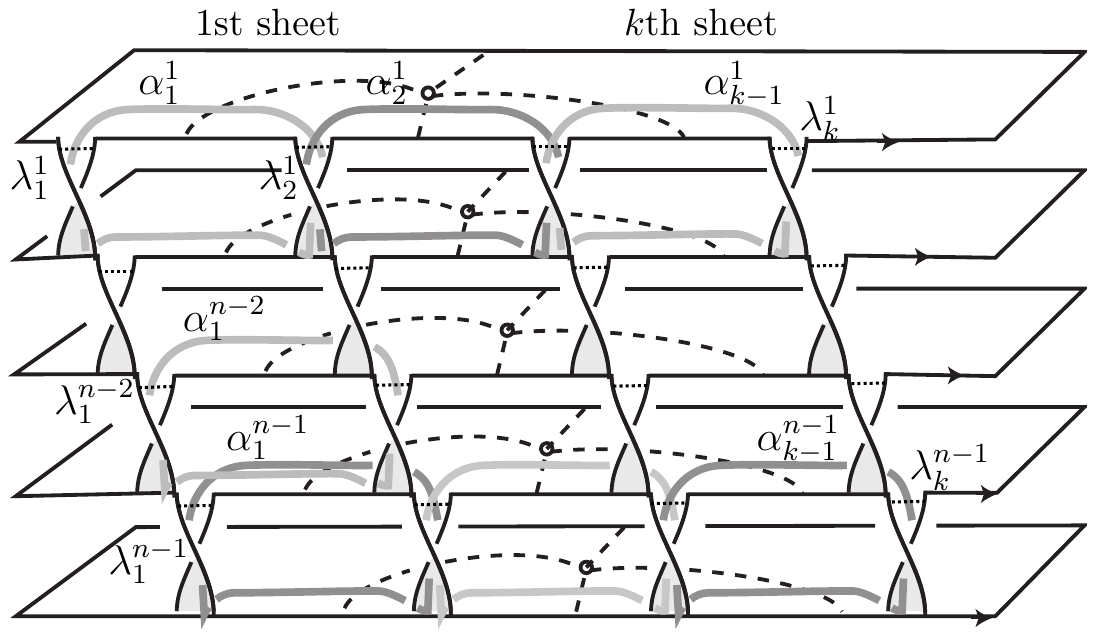}
\caption{The Seifert surface $S$ of a $(n,k)$ torus link, where $n=5$ and $k=4$. Hollowed circles are the lifted branch points $\widetilde P$. Cutting $S$ along the dashed arcs gives $k$ sheets of disks. The deck transformation $\iota_k$ takes $i$-th disk to $i+1$st disk. 
} 
\label{torus-link}
\end{figure}

Let $\beta \in B_{n} \approx \Mod(D\setminus P)$ and $f_{\beta}:(D,P) \rightarrow (D,P)$ be a homeomorphism representing the braid $\beta$. 
Since $\beta(\ker (e_k))= \ker(e_k)$ via the action of $B_{n}$ on $\pi_{1}(D \setminus P)$,
there is a lift $\widetilde{f_{\beta}}: S \rightarrow S$ of $f_{\beta}$ (see \cite[Lemma 5.1]{BH} and \cite[Theorem 1.1]{GW}).
Among the $k$ possible lifts which are related to each other by deck transformations, $\widetilde{f_{\beta}}$ is the unique lift that fixes $\partial S$ pointwise.

When two homeomorphisms $f_{\beta}$ and $f'_{\beta} \in \Homeo(D, P)$ represent the same braid $\beta\in B_n$ we note that an isotopy between $f_{\beta}$ and $f'_{\beta}$  lifts to an isotopy between their lifts $\widetilde{f_{\beta}}$ and $\widetilde{f'_{\beta}}$. Thus, we have a well-defined homomorphism
\[ \Psi :B_{n} \rightarrow \Mod(S) \]
defined by $\Psi(\beta) = [\widetilde{f_{\beta}}]$.
Let $$\Theta=\Theta_{k}: \Mod(S) \rightarrow \Mod(S)$$ be an automorphism defined by $ \Theta([f]) =[\iota_k \circ f \circ \iota_k^{-1}]$.

\begin{definition}
Define
\begin{equation*}
\SMod(S)  :=  \left\{ [f] \in \Mod(S) \ \middle| \ f \in \Homeo(S) \mbox{ and } \Theta([f])=[f] \right\} 
\end{equation*}
and call it the \emph{symmetric mapping class group}.
An element of $\SMod(S)$ is called {\em symmetric mapping class}. 
\end{definition}

The following fact is attributed to Birman and Hilden. 

\begin{proposition}\label{prop-Birman-Hilden} 
The homomorphism $\Psi$ is injective and onto $\SMod(S)$. 
\end{proposition} 

For the sake of completeness, we include a proof of Proposition~\ref{prop-Birman-Hilden}. 

\begin{proof}[Proof of Proposition~\ref{prop-Birman-Hilden}] 
The injectivity of $\Psi$ follows from \cite{BH}.
A homeomorphism $f: S \rightarrow S$ is called \emph{fiber-preserving} if $\pi\circ f(p) = \pi \circ f(p')$ for any $p,p' \in S$ with $\pi(p)=\pi(p')$.
Suppose that $f_0, f_1 \in \Homeo(D, P)$ satisfy $\Psi([f_0])=\Psi([f_1])$. 
According to \cite[Theorem 1]{BH},
since the lifts $\tilde{f_0}, \tilde{f_1}\in\Homeo(S)$ represent the same element of $\SMod(S)$ and are fiber-preserving,   
there exists a fiber-preserving isotopy $g_t\in \Homeo(S)$ between $\tilde{f_0}$ and $\tilde{f_1}$.
Since $g_t$ is fiber-preserving the composition $\pi\circ g_t\circ \pi^{-1}$ is a well-defined homeomorphism of $(D, P)$ and it gives an isotopy between $f_0$ and $f_1$; hence, $[f_0]=[f_1] \in B_n$.

To see $\im(\Psi)\supset\SMod(S)$, assume that $f\in \Homeo(S)$ satisfies $[f] = [\iota \circ f\circ \iota^{-1}]$. 
The same argument as in the proof of \cite[Theorem 3]{BH} with Teichm\"uller's theorem for bordered surfaces \cite[page 59]{A} shows that $[f]$ can be represented by a homeomorphism $f' \in \Homeo(S)$ such that $f' = \iota \circ f'\circ \iota^{-1}$.

For $x \in D$ let $\tilde x \in \pi^{-1}(x) \subset S$ be a pre-image of $x$ under the branched covering map $\pi:S\to D$. 
Define a homeomorphism $b \in \Homeo(D)$ by 
$$b(x)=\pi(f'(\tilde x)).$$
Since $\pi\circ f'\circ \iota(\tilde x)=\pi\circ\iota\circ f'(\tilde x)=\pi\circ f'(\tilde x)$, the image $b(x)$ is independent of the choice of the pre-image $\tilde x$.

We observe that $b$ acts on the branch set $P$ as a permutation.
Suppose that $p \in P$ and $\tilde p \in \tilde P$ satisfy $\pi(\tilde p)=p$. 
Since $\iota$ fixes the branch set $\tilde P$ point-wise we have $\iota(f'(\tilde p))=f'(\iota (\tilde p))=f'(\tilde p)$.
That is, $f'(\tilde p)$ is a fixed point of $\iota$ and $f'(\tilde p) \in \tilde P$. 
We get $b(p)= \pi(f'(\tilde p)) \in \pi (\tilde P)=P$ which shows $b(P)=P$.

Since $\pi\circ f' = b \circ \pi$ the map $f'\in\Homeo(S)$ is the unique lift of $b\in\Homeo(D;P)$. We obtain $[f]=[f']=\Psi([b])\in \im(\Psi)$.

To see $\im(\Psi)\subset\SMod(S)$, let $\alpha^i_l$ be an simple closed curve that goes through $i$th and $(i+1)$th sheets and $l$th and $l+1$st twisted bands as depicted in Figure~\ref{torus-link}. 
Let $t_{i,l}:=t_{\alpha^{i}_{l}}\in\Homeo(S)$ be a positive Dehn twist about $\alpha^i_l$. 
It is shown in \cite[Lemma~3.1]{HKP} that for the standard braid generators $\sigma_1,\dots,\sigma_{n-1}$ of the Artin braid group $B_n$ we have 
$$\Psi(\sigma_i) =[t_{i,1} \circ t_{i,2} \circ \dots \circ t_{i,k-1}] \in \Mod(S).$$
For $j=1,\dots,n-1$ and $l=1,\dots, k$ let $\lambda^j_l$  (see Figure~\ref{torus-link}) be properly embedded arcs on the $l$th band between $j$th and $j+1$st sheets such that $\iota(\lambda^j_l)=\lambda^j_{l+1}$. 
Let $[\lambda^j_l]$ denote the isotopy class of $\lambda^j_l$ relative to $\partial S$. 
We get 
$$
\left[\iota^{-1}\circ (t_{i,1} \circ t_{i,2} \circ \dots \circ t_{i,k-1}) \circ\iota\right] \left[\lambda^j_l\right] =
\left[t_{i,1} \circ t_{i,2} \circ \dots \circ t_{i,k-1}\right] \left[\lambda^j_l\right]
$$ 
for all $j=1,\dots,n-1$ and $l=1,\dots, k$. 
Knowing that the surface $S$ cut along the union of arcs $\cup_{j, l} \lambda^j_l$ yields $n$ disjoint disks, the Alexander trick \cite[Proposition 2.8]{FM} implies that $\left[\iota^{-1}\circ (t_{i,1} \circ t_{i,2} \circ \dots \circ t_{i,k-1}) \circ\iota\right]  =
\left[t_{i,1} \circ t_{i,2} \circ \dots \circ t_{i,k-1}\right]$. 
This shows that $\Psi(\sigma_i)\in\SMod(S)$. 
\end{proof}

\section{Positive answers for simple cases}
\label{section:panswer}

\begin{theorem}
\label{theorem:main1}
We have 
$P(2)=\QP(2) = \Dehn^+(2,k)= Tight(2,k) = RV(2)$ for all $k \geq 2$.
\end{theorem}

\begin{proof}
The first equality $P(2)=\QP(2)$ is obvious. 

By Proposition \ref{subset} we only need prove $\QP(2) \supset\Dehn^+(2,k)$.
Take a braid $b \in B_2 \setminus \QP(2)$. 
There exists a positive integer $m$ such that $b = (\sigma_1)^{-m}$. 
Let $T_i = [t_{\alpha^1_i}] \in \Mod(S)$. 
By Lemma 3.1 of \cite{HKP} we have $\Psi(b)=(T_1 \circ T_2 \circ \dots \circ T_{k-1})^{-m}$. 
Clearly $(T_1 \circ \dots \circ T_{k-1})^{-m} \notin Veer^+(S)$. 
Since $Dehn^+(S) \subset Veer^+(S)$ we have $\Psi(b)\notin Dehn^+(S)$; hence, $b \notin  \Dehn^+(2,k).$
\end{proof}

\begin{theorem}
\label{theorem:main2}
For $n\leq 4$, $\QP(n) = \Dehn^+(n,2)$.
\end{theorem}

The equality $QP(3)=Dehn^+(3, 2)$ has been known and used in the literature. However for sake of completeness, we give a proof of this case along with the case of $n=4$. 

\begin{proof}
Since the $n=2$ case is shown in Theorem \ref{theorem:main1} we may assume $n=3$ or $4$.

Let $a_1,a_{2},\delta$ (for the $n=3$ case) and $a_1,a_2,a_3,\delta_1, \delta_{2}$ (for the $n=4$ case) be simple closed curves on $S$ as depicted in Figure~\ref{fig:mcggenerator}. 
The Dehn twists about these curves generate $\Mod(S)$, and $\Psi(\sigma_i)=T_{a_i}$ holds.
\begin{figure}[htbp]
\begin{center}
\includegraphics*[bb=71 644 328 734, width=80mm]{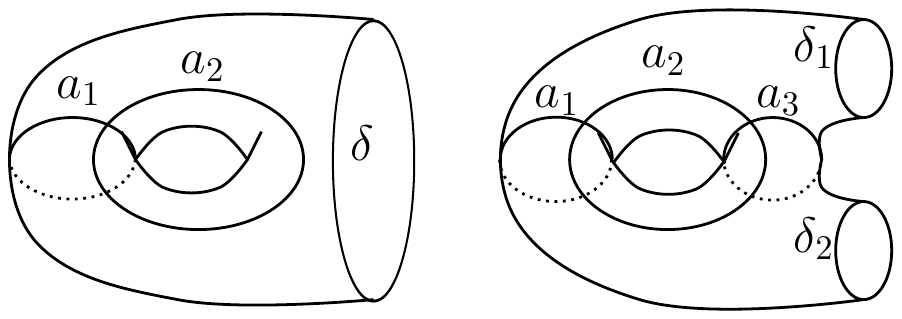}
\end{center}
\caption{}
\label{fig:mcggenerator}
\end{figure}

If $b \in Dehn^{+}(n,2)$ then there exist non-boundary parallel simple closed curves $\Gamma_1,\dots,\Gamma_m \subset S$ (possibly $\Gamma_i = \Gamma_j$ for $i\not=j$) and $x, y \geq 0$ such that 
\begin{equation*}
\Psi(b) = 
\begin{cases}
T_{\Gamma_1} \circ\dots\circ T_{\Gamma_m} \circ T_{\delta}^{x} & (n=3) \\
T_{\Gamma_1} \circ\dots\circ T_{\Gamma_m} \circ T_{\delta_1}^x \circ T_{\delta_2}^y & (n=4).
\end{cases}
\end{equation*}
We may assume that all of the curves $\Gamma_i$ are non-separating since the positive Dehn twist about a separating curve can be written as a product of the positive Dehn twists about non-separating curves. 
Since $a_1$ is non-separating $\Gamma_i = f_i(a_1)$ for some $f_i \in \Mod(S)$. We can write $f_i$ as
\[ f_i = T_{a_{j_l}}^{\e_l}\circ \cdots \circ T_{a_{j_1}}^{\e_1} \]
for some $j_p \in \{1,\ldots, n-1\}$ and $\e_p \in \{\pm 1\}$. 
Since the Dehn twist along the boundary does not affect simple closed curves in $S$, the words $T_{\delta}$, $T_{\delta_{1}}$ and $T_{\delta_2}$ are not needed in the expression.
Note that $f_{i} = \Psi(b_{i})$ for $b_{i}=\sigma_{j_l}^{\e_l}\cdots \sigma_{j_1}^{\e_1} \in B_{n}$.
Therefore, we have
\[
T_{\Gamma_i} = T_{f(a_1)}  =  f_i \circ T_{a_1} \circ f_i^{-1} = \Psi(b_i) \Psi(\sigma_{1}) \Psi(b_{i}^{-1}) = \Psi(b_{i}\sigma_1 b_{i}^{-1})
\]
and $T_{\Gamma_i} \in \Psi(QP(n))$ for all $i$.
In particular, $\Theta(T_{\Gamma_1} \circ\dots\circ T_{\Gamma_m}) = T_{\Gamma_1} \circ\dots\circ T_{\Gamma_m}$ by Proposition~\ref{prop-Birman-Hilden}.

For the case $n=4$, since $\iota(\delta_1)=\delta_2$, $\Theta(T_{\delta_{1}}) = T_{\delta_{2}}$ and $\Theta(T_{\delta_{2}}) = T_{\delta_{1}}$ which give\begin{eqnarray*}
T_{\Gamma_1} \circ\dots\circ T_{\Gamma_m}\circ T_{\delta_1}^x \circ T_{\delta_2}^y  
&=&
\Theta( T_{\Gamma_1} \circ\dots\circ T_{\Gamma_m} \circ T_{\delta_1}^x \circ T_{\delta_2}^y) \\
&=& 
\Theta(T_{\Gamma_1} \circ\dots\circ T_{\Gamma_m}) \circ\Theta(T_{\delta_1}^x) \circ \Theta(T_{\delta_2}^y)  \\
&=& 
T_{\Gamma_1} \circ\dots\circ T_{\Gamma_m}\circ  T_{\delta_2}^x \circ T_{\delta_1}^y.
\end{eqnarray*}
Therefore, $x=y$.

With the chain relations \cite[Proposition 4.12]{FM} $T_{\delta} = \Psi((\sigma_{1}\sigma_{2})^{6})$ and $T_{\delta_1}^x T_{\delta_2}^x = \Psi((\sigma_1\sigma_2\sigma_3)^{4x})$,
we can conclude $b \in QP(n)$.
\end{proof}

\section{Roots in quasi-positive braids}
\label{section:QP}

In this section we address the following question. Among the results, Corollary~\ref{prop:rootqperidic} plays an important role to prove our main theorems in Section~\ref{section:suff}.

\begin{question}
\label{question:QProot}
Are $QP(n)$ and $Dehn^{+}(n,k)$ closed under taking roots? 
That is, 

(1) Does $b=x^d \in QP(n)$ for an integer $d\geq 2$ imply $x \in QP(n)$?

(2) Does $x^d \in Dehn^{+}(n,k)$ for an integer $d\geq 2$ imply $x\in Dehn^{+}(n,k)$? 
\end{question}

\begin{definition}[Property {\bf(QP-root)}]
We say that a quasi-positive braid $b \in QP(n)$ has the property {\bf (QP-root)} if the following condition is satisfied. 
\[ \mbox{ If } b=x^{d} \mbox{ for } x \in B_{n} \mbox{ and } d \geq 2 
\mbox{ then }  x \in QP(n). \]
\end{definition}

It is shown  in \cite[Theorem 1.1]{GM} that the $d$-th root of a braid (if it exists) is unique up to conjugacy; namely, $x^{d}=y^{d} \in B_n$ implies that $x$ and $y$ are conjugate to each other. This leads to the following observation.

\begin{proposition}
\label{proposition:root}
For $x, x' \in B_{n}$ assume that $x^{d} = x'^{d}$.
Then $x \in QP(n)$ if and only if $x' \in QP(n)$.
\end{proposition}

Let $b\in B_n$ be a non-periodic reducible braid. According to \cite[Definition 5.1]{GMW}, up to conjugacy, $b$ is in  the following {\em regular form}.

Let $C$ be an essential multi-curve in the $n$-punctured disk $D_n$ such that 
\begin{itemize}
\item
$b(C)$ is isotopic to $C$, and 
\item
any simple closed curve which has non-zero geometric intersection with $C$ must not be preserved by any power of $b$. 
\end{itemize}
Such a multi-curve $C$ is called a {\em canonical reduction system} for $b$ \cite{BLM} and it is unique up to conjugacy. 
It always exists for a non-periodic reducible braid \cite{I}.

Let $\mathcal A' $ be the set of outermost curves of $C$.
There could exist punctures in $D_n$ not enclosed by any circle in $\mathcal A'$. 
Define a set of curves $\mathcal A$ to be the union of $\mathcal A'$ and one circle around each such puncture of $D_n$. 
We may enumerate the elements of $\mathcal A$ as 
$$\mathcal{A}=\{ a_{1,1},\ldots, a_{1,r_1}, a_{2,1},\ldots, a_{2,r_2}, \dots, a_{c,1},\ldots, a_{c,r_c}\}$$ 
so that $b(a_{i,j}) = a_{i,j+1}$ $(j = 1,\ldots,r_i-1)$ and $b(a_{i,r_i}) = a_{i,1}$. 
This action of $b$ on $\mathcal A$ gives $m=r_1+r_2+\cdots+r_c$ disjoint braided tubes (see Figure~\ref{figure:redbraid} (i)). 
Identifying each tube with a string, we get an $m$-braid $\widehat{b} \in B_m$ which we call the {\em tubular braid} associated to $\mathcal A$.
By the nature of the canonical reduction system, this braid $\widehat b$ is unique up to conjugacy.  
\begin{figure}[htbp]
\begin{center}
\includegraphics*[width=120mm, bb=70 533 441 735]{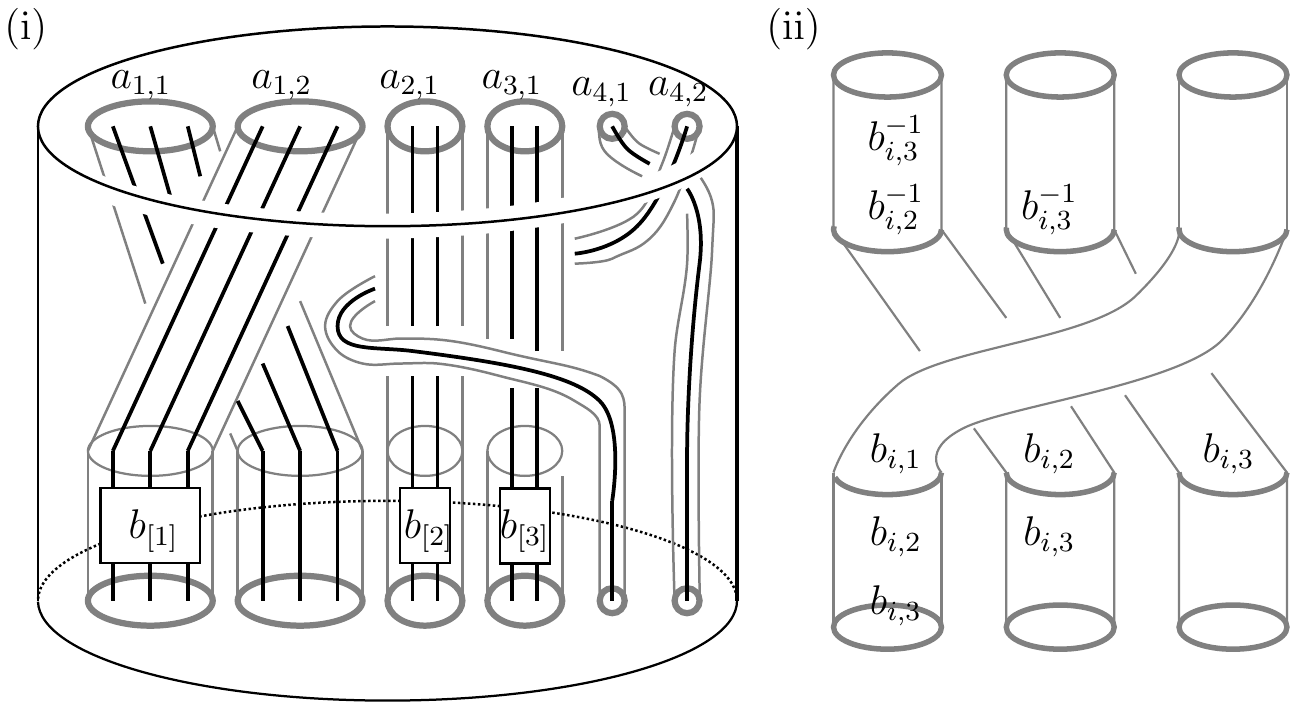}
\end{center}
\caption{ (i) Non-periodic and reducible braid in regular form. (ii) Taking a conjugate the condition (\ref{eqn:regularform}) is  satisfied.}
\label{figure:redbraid}
\end{figure}

The braid closure of $\widehat b$ gives a $c$-component link. 
The braid closure of the original braid $b$ can be viewed as a satellite of the $c$-component link. 
Let $b_{i,j}$ denote the braiding inside the tube which starts at $a_{i,j}$ and ends at $a_{i,j+1}$ (where the indices $j$ are considered up to modulo $r_{i}$), which we call the \emph{interior braids}.

We say that $b$ is in \emph{regular form} if
\begin{equation}
\label{eqn:regularform}
\mbox{the only non-trivial interior braids are } b_{1,r_1},\ldots, b_{c,r_c}.  
\end{equation}
We denote these non-trivial interior braids by $b_{[1]},\ldots,b_{[c]}$.
The condition (\ref{eqn:regularform}) can be realized by shifting non-trivial interior braid $b_{i,j}$ $(j \neq r_i)$ along the closure of the tubular braid $\widehat b$, which is equivalent to taking a conjugate (see Figure \ref{figure:redbraid} (ii)). 
In this process, the tubular braid $\widehat{b}$ remains the same and 
\begin{equation}
\label{eqn:intbraid}
b_{[i]} = b_{i,1}b_{i,2}\cdots b_{i,r_i} \mbox{ (up to conjugacy).}  
\end{equation}

\begin{remark}
In \cite{GMW} a regular form requires one more property that if $b_{[i]}$ and $b_{[j]}$ are conjugate then $b_{[i]} = b_{[j]}$. But we do not use this property in this paper. 
\end{remark}

Here is a simple observation. 
\begin{lemma}\label{lem}
If a braid $b$ is in regular form with quasi-positive tubular braid and interior braids then $b$ is also quasi-positive. 
\end{lemma}

The converse of Lemma~\ref{lem} is not true. 
The reducible $4$-braid $(\sigma_{2}\sigma_{3}\sigma_{2}^{-1})(\sigma_{1}\sigma_{2}\sigma_{1}^{-1}) = (\sigma_{2} \sigma_{1}\sigma_{3}\sigma_{2})\sigma_{1}^{-2}$ is quasi-positive and in regular form with the tubular braid $\sigma_1 \in B_2$ and the interior braid $\sigma_{1}^{-2}\in B_2$.

The next proposition gives a criterion of the property {\bf (QP-root)}.

\begin{proposition}\label{prop:reducible}
Let $b$ be a non-periodic reducible, quasi-positive braid. If $b$ has a regular form such that all of its tubular and interior braids are quasi-positive and having the property {\bf (QP-root)}, then $b$ also has the property {\bf (QP-root)}.
\end{proposition}

\begin{proof}
Let $b$ be a non-periodic reducible, quasi-positive braid. 
Assume that $b=x^{d}$ for some $x \in B_{n}$ and $d\geq 1$. 
We will show that $x$ is quasi-positive. 
Suppose that $b$ is in regular form with quasi-positive tubular braid $\widehat b \in B_m$ and quasi-positive interior braids $b_{[1]},\cdots, b_{[c]}$ all of which have the property {\bf (QP-root)}.

Since $b$ is non-periodic and reducible, the root $x$ is also non-periodic and reducible. 
We may assume that $x$ is in regular form with tubular braid $\widehat x \in B_m$ and interior braids $x_{[1]},\ldots, x_{[c']}$.  
Since $b=x^{d}$ we see that $c'\leq c$ and $\widehat b$ is  conjugate to $(\widehat x)^d$. 
That is, there exists $y \in B_m$ such that $\widehat b = y (\widehat x)^d y^{-1} = (y \widehat x y^{-1})^d$. 
By the property {\bf (QP-root)} of $\widehat b$, the tubular braid $y\widehat x y^{-1}$ (hence, $\widehat x$) is quasi-positive.

As for the interior braids, each $b_{[i]}$ is conjugate to a power of some single interior braid $x_{[i']}$. 
Moreover, for each $i' \in\{1,\ldots, c'\}$ there exists $i\in\{1,\ldots, c\}$ such that $b_{[i]}$ is conjugate to a power of  $x_{[i']}$.
Therefore, by the property {\bf (QP-root)} of $b_{[1]},\ldots,b_{[c]}$, all the interior braids $x_{[1]},\ldots, x_{[c']}$ are quasi-positive. 

By Lemma~\ref{lem} we conclude that $x$ is quasi-positive. 
\end{proof}

A classical theorem of Ker\'ej\'art\'o and Eilenberg states that any  root of a periodic $n$-braid is conjugate to either
$(\sigma_1 \sigma_2 \sigma_3 \cdots \sigma_{n-1})^{i}$ or $(\sigma_{1}^{2}\sigma_2 \sigma_3 \cdots \sigma_{n-1})^{i}$ for some $i$. Thus, every periodic braid has the property {\bf (QP-root)}. 
This fact and Proposition~\ref{prop:reducible} imply the following corollary.

\begin{corollary}
\label{prop:rootqperidic} 
Let $b\in QP(n)$ admitting a factorization $b= T_{a_1}^{N_1} \circ \cdots \circ T_{a_p}^{N_p}\in \Mod(D_n)=B_n$ where $a_1,\ldots,a_p$ are pairwise disjoint simple closed curves in $D_{n}$ and $N_1, \cdots, N_p >0$.
If $b=x^d$ for some $d\geq 2$ then $x \in QP(n)$. 
\end{corollary}

The above results show that Question \ref{question:QProot}, whether a quasi-positive braid has the property {\bf (QP-root)}, is reduced to pseudo-Anosov braids.

The next proposition gives a potential negative answer to  Question \ref{question:QProot}.

\begin{proposition}
\label{prop:counterexample}
Let 
$b \in QP(n)$ be a pseudo-Anosov quasi-positive braid.
Let $Ab: B_{n} \rightarrow \Z$ denote the abelianization (exponent sum) map. 
If $b=x^d$ for some $d>1$ and $x\in B_n$ with $Ab(x)=1$ 
then $b$ does not have the property {\bf (QP-root)}. 
\end{proposition} 
\begin{proof}
Since $Ab(x)=1$, if $x$ is quasi-positive then $x$ must be the positive half twist about a simple arc connecting two distinct punctures of $D_{n}$. 
Such a mapping class $x$ is reducible. Hence, $x^d$ is reducible and $b=x^d$ cannot be pseudo-Anosov. 
\end{proof}

At this time of writing, we do not know whether a pseudo-Anosov quasi-positive braid that satisfies the assumption of Proposition~\ref{prop:counterexample} exists or not.

\section{Sufficient conditions for quasi-positive}
\label{section:suff}

In this section, we study quasi-positive braids admitting certain  symmetric conditions. 
Throughout this section, we fix a hyperbolic structure on $S$ so that the deck transformation $\iota=\iota_k:S \rightarrow S$ is an isometry (see \cite[Theorem 11.6]{FLP}). 
Let $S' \subsetneq S$ be a connected subsurface of $S$ that satisfies one of the following conditions. 
\begin{itemize}
\item $S'$ is an annular neighborhood of a geodesic simple closed curve in $S$.
\item $S'$ is a hyperbolic surface with geodesic boundary and the inclusion $S' \hookrightarrow S$ is an isometry.  
\end{itemize}
Note that the surface $S'_{i}:=\iota^{i-1}(S') \subset S$ also satisfies the same property.
For $f \in \Homeo(S')$ let $\widehat{f} \in\Homeo(S)$ be a homeomorphism extending $f$ such that $\widehat{f}(x)=x$ for $x \in S \setminus S'$. 
For $i=1,\ldots,k$ let $$f_i:=\iota^{i-1} \circ \widehat{f} \circ \iota^{-i+1} \in \Homeo(S).$$ 
Our goal is to study elements $\Psi(b)=[\phi] \in \SMod(S) \cap Dehn^+(S)$ of the form
\[ \phi=f_{1}f_{2}\cdots f_k \in \Homeo(S)\] 
and find sufficient conditions 
that guarantees $b \in QP(n)$.

We first study the following special case.

\begin{theorem}
\label{theorem:twist}
Let $C$ be a simple closed geodesic curve in $S$ such that $C, \iota(C), \dots, \iota^{e-1}(C)$ are pairwise disjoint with $\iota^e(C)=C$ for some $e\in\{1,\dots,k\}$ that divides $k$. 
Let $d, j \in \N$. 
Suppose that $b^d \in \Dehn^+(n, k)$
with  
$$
\Psi(b^d)= \left(T_C \circ T_{\iota(C)} \circ T_{\iota^2(C)} \circ\dots\circ T_{\iota^{e-1}(C)}\right)^j.$$ 
Then $b \in QP(n)$ $($and so $b^d \in QP(n)).$
\end{theorem}

When $e=1$ we get the following.
\begin{corollary}
\label{cor:froot}
Let $j,d\in\N$. 
If $b^d \in Dehn^+(n, k)$ with $\Psi(b^d)= T_C^j$ then $b \in QP(n)$. 
\end{corollary}

By Proposition~\ref{prop-Birman-Hilden}  it is easy to see that $\Psi(b^d)= T_C^j$ implies $\iota(C)=C$.

\begin{proof}[Proof of Theorem~\ref{theorem:twist}]

Since $C$ is simple and $\pi^{-1}(\pi(C))=C \sqcup \iota(C) \sqcup \cdots \sqcup \iota^{e-1}(C)$, the projection $\pi(C)$ is also simple. 

First, we treat an exceptional case where the projection $\pi(C)$ is a simple proper arc in the $n$-punctured disk $D_{n}:=D \setminus P$ connecting two distinct punctures. This can be realized only if $k=2$, $e=1$ and $C$ is a non-separating simple closed curve in $S$. 
Let $h \in QP(n)$ be the braid represented by a positive half twist about the arc $\pi(C)$. We have $\Psi(h) = T_{C}$. 
Thus, $\Psi(b^{d})=T_{C}^{j} = \Psi(h^{j})$.
Since $\Psi$ is injective (Proposition~\ref{prop-Birman-Hilden}) $b^{d}=h^{j} \in QP(n)$. Let $Ab : B_{n} \rightarrow \Z$ be the abelianization map defined by $Ab (\sigma_i^{\pm 1})=\pm 1$.
Since $Ab (h)=1$ we get $Ab(b) \cdot d = j$ and $\Psi(b^{d}) = \Psi(h^j)= \Psi((h^{Ab(b)})^{d})$. 
Proposition \ref{proposition:root} implies that $b \in QP(n)$.

Next, suppose that the projection $\pi(C)$ is a simple closed curve in the punctured disk $D_n$. 
Let $k' = \frac{k}{e}$. Since the map $\pi$ restricted to each connected component of $\pi^{-1}(\pi(C))$ is a $k':1$ cover we have 
\begin{equation}
\label{eqn:liftC}
\Psi((T_{\pi(C)})^{k'}) = T_C \circ T_{\iota(C)} \circ T_{\iota^2(C)} \circ\dots\circ T_{\iota^{e-1}(C)}.
\end{equation}
Hence, 
\[ \Psi(b^{d}) = (T_C \circ T_{\iota(C)} \circ T_{\iota^2(C)} \circ\dots\circ T_{\iota^{e-1}(C)})^j = \Psi( (T_{\pi(C)})^{k' j}).\]
Proposition~\ref{prop-Birman-Hilden} gives $b^{d}= (T_{\pi(C)})^{k'j}$. By Corollary~\ref{prop:rootqperidic}, $b \in QP(n)$.
\end{proof}

Next, we study a more general case.
Recall that $f \in \Homeo(S')$, $S'_{i}:=\iota^{i-1}(S') \subset S$,  and   $f_i:=\iota^{i-1} \circ \widehat{f} \circ \iota^{-i+1} \in \Homeo(S)$ for $i=1,\ldots,k$.

\begin{lemma}\label{lemma:centralizer}
Suppose that $[f]\in \Mod(S')$ is pseudo-Anosov. 
Any centralizer $[g]\in Z([f]) \subset \Mod(S')$ of $[f]$ is either pseudo-Anosov or periodic 
\end{lemma}

\begin{proof}
Suppose that $[f] \in\Mod(S')$ is pseudo-Anosov and $[g]\in Z([f])$. 
Let $\mathcal F$ be the stable foliation of $[f]$. 
Since $\mathcal F$ is preserved under $[f]$ we have 
$[f][g](\mathcal F) = [g][f](\mathcal F)=[g](\mathcal F)$, which means that the foliation $[g](\mathcal F)$ is either $\mathcal F$ itself or the unstable foliation of $[f]$. In either way, the homeomorphism $g\in\Homeo(S')$ is freely isotopic to a pseudo-Anosov map or a periodic map.  
\end{proof}

Here is our main result. 
Later in Example~\ref{example-daisy} we see that it is generalizing the so called {\em daisy relation } \cite{EMV} in mapping class group theory.

\begin{theorem}\label{thm:general}
Suppose that the surfaces $S'_1,\ldots, S'_k$ are pairwise non-isotopic. 
Assume that $[f]\in Dehn^{+}(S')$ is either 
\begin{itemize}
\item a non-negative power of a single Dehn twist (i.e., $S'$ is an annulus which is neighborhood of a simple closed geodesic curve), or 
\item a pseudo-Anosov map (i.e., $S'$ is a hyperbolic surface).
\end{itemize} 
Suppose that $b \in Dehn^+(n,k)$ satisfies 
\[ \Psi(b)=[f_1 f_2 \cdots f_{k}] \]
then $b \in QP(n)$.
\end{theorem}

\begin{proof}

There are two cases to consider.

\noindent
{\bf Case 1:} $S'_1,S'_2,\ldots,S'_k$ are pairwise disjoint.

\noindent
Since $[f] \in Dehn^{+}(S')$ we may write $[f_1]=T_{C_1}\cdots T_{C_l}$ for some simple closed curves $C_1,\ldots, C_l \subset S'$. 
We get $[f_{j}]=[\iota^{j-1} \widehat{f} \iota^{-j+1}]= T_{\iota^{j-1}(C_1)}\cdots T_{\iota^{j-1}(C_l)}$.
Since $S'_1,S'_2,\ldots,S'_k$ are pairwise disjoint, if $j\neq j'$ then $T_{\iota^{j}(C_{i})}$ and $T_{\iota^{j'}(C_{i'})}$ commute for every $i,i'$ and we have 
\begin{eqnarray*}
\Psi(b) & = & [f_{1}\cdots f_{k}] \\
&= &
 (T_{C_1}\cdots T_{C_l})(T_{\iota(C_1)}\cdots T_{\iota(C_{l})}) \cdots (T_{\iota^{k-1}(C_1)}\cdots T_{\iota^{k-1}(C_l)})\\
& = &(T_{C_1} T_{\iota(C_1)}\cdots T_{\iota^{k-1}(C_{1})}) \cdots (T_{C_l} T_{\iota(C_l)}\cdots T_{\iota^{k-1}(C_{l})}).
\end{eqnarray*}
Since $S'_1,S'_2,\ldots,S'_k$ are pairwise disjoint, the projection $\pi(C_i)$ is a simple closed curve in $D_n$, and we can use the second-half argument in the proof of Theorem~\ref{theorem:twist}.
By the formula (\ref{eqn:liftC}) we get $\Psi(T_{\pi(C_{i})})=T_{C_i} T_{\iota(C_i)}\cdots T_{\iota^{k-1}(C_{i})}$. 
Proposition~\ref{prop-Birman-Hilden} gives $b= T_{\pi(C_1)}\cdots T_{\pi(C_l)} \in QP(n)$.

\noindent
{\bf Case 2:} $S'_1 \cap S'_p \neq \emptyset$ for some $p\in\{2,\dots,k\}$.

\noindent
First we note that 
$[f_1 f_2 \cdots f_{k}]=\Psi(b) \in \SMod(S)$ implies that 
\[ 
[f_1f_2\cdots f_k]  = [\iota( f_1 f_2 \cdots f_k) \iota^{-1}] = [f_{2}f_3\cdots f_k f_1] = \dots = [f_k f_1\cdots f_{k-1}]. 
\]
In particular, we have
\begin{equation}
\label{eqn:center} 
\Psi(b) [f_i] = [f_i] \Psi(b) \mbox{ for every } i\in\{1,\dots,k\}. 
\end{equation}

Let $\nu$ be the minimal subsurface of $S$ such that
\begin{itemize}
\item $\nu$ contains $S'_1 \cup S'_2 \cup \cdots \cup S'_k$ and 
\item $\partial \nu$ is geodesic.
\end{itemize}
Since $\iota(S'_1 \cup S'_2 \cup \cdots \cup S'_k) = S'_1 \cup S'_2 \cup \cdots \cup S'_k$ we have $\iota(\nu)=\nu$. 
In particular, the boundary $\partial \nu$ is a multi-curve invariant under $\iota$.

For simplicity, we put $\psi:= f_1 f_2 \cdots f_k \in \Homeo(S)$. 

\begin{claim}
\label{claim:subsurface}
$\psi(S'_i)$ is isotopic to $S'_i$. 
\end{claim}

We will prove Claim~\ref{claim:subsurface} after the completion of the proof of Theorem~\ref{thm:general}. 

By Claim~\ref{claim:subsurface} 
there is a homeomorphism $\phi \in \Homeo(S)$ which is isotopic to $\psi$ and preserves $S'_i$ setwise. 
Although $\phi$ may permute components of $\partial S'_{i}$, we may assume that there exists $d_0>0$ such that $\phi^{d_0}=id$ on $\partial S'_1 \cup \partial S'_2 \cup \cdots \cup \partial S'_k$.
Let 
\begin{equation}\label{def of phi_i}
\phi_i := \phi^{d_0}|_{S'_{i}} \in \Homeo(S'_i).
\end{equation}

\begin{claim}
\label{claim:btwist}
For each $i=1,\ldots,k$ there is $d_i>0$ such that $\phi_i^{d_i}$ is isotopic to a  product of Dehn twists about the boundary components of $S'_i$. Namely, 
\begin{equation}\label{eqn}
[\phi_i^{d_i}] = \prod_{C \subset \partial S'_i} T_{C}^{N(C)} \in \Mod(S'_i) \quad \mbox{ for some } N(C) \in\Z.
\end{equation}
Moreover, the exponent $N(C)=0$ if $C$ is essential in the minimal subsurface $\nu$. 
Therefore, if $S'$ is an annulus then $[\phi_i^{d_i}] = id$.
\end{claim}

We will prove the claim after the completion of the proof of Theorem~\ref{thm:general}.

Let $d'$ be the least common multiple of $d_1,\ldots,d_k$ found in Claim~\ref{claim:btwist}. Put $d=d_{0}d'$.
As a consequence of Claim~\ref{claim:btwist} the map $\phi^{d}|_{S_{i}} = \phi_{i}^{d'} \in \Homeo(S'_i)$ is isotopic to a product of Dehn twists along common components of $\partial S'_i$ and  $ \partial \nu$. 
Let $C_{1},\ldots,C_m$ denote the boundary components of $\nu$.
We obtain 
\begin{equation}\label{phi_d}
[\phi^{d}] = [T_{C_{1}}^{N(1)}T_{C_{2}}^{N(2)}\cdots T_{C_{m}}^{N(m)}]. 
\end{equation} 
Since $[f] \in Dehn^{+}(S')$ we see that $[\phi]=[f_1\cdots f_k]$ is right-veering and $N(i)\geq 0$ for all $i=1,\dots,m$. 

We define an equivalence relation $C_{i} \sim C_{i'}$ if $\pi(C_{i})=\pi(C_{i'}) \subset D_{n}$ and let 
$\mathcal{C} = \{C_{1},\ldots,C_m\} \slash \sim$. 
For $C \in \{C_1, C_2,\dots, C_m\}$ let $[C]\in\mathcal C$ be its equivalence class and $e(C)\in\N$ the smallest positive integer such that $\iota^{e(C)}(C)=C$. 
Since $\iota(\partial \nu)=\partial\nu$ we note that $C, \iota(C), \dots, \iota^{e(C)-1}(C) \subset \partial \nu = C_1 \sqcup C_2\sqcup \cdots \sqcup  C_m$.
Put $$T_{[C]} := T_{C}T_{\iota(C)}\cdots T_{\iota^{e(C)-1}(C)}.$$ 
The fact $[\phi] \in \SMod(S)$ implies that $N(i)=N(j)$ if $C_{i} \sim C_j$. We may define non-negative integers  $N([C_i]):=N(i)$. 
The description (\ref{phi_d}) can be restated as 
\[ 
[\phi^{d}] = \prod_{[C] \in \mathcal{C}} T_{[C]}^{N([C])} \in \Mod(S). 
\]

Recall that the projection $\pi(C)$ is a simple closed curve or a simple proper arc in $D_n$ joining distinct punctures. 
This is because $\pi^{-1}\pi(C) \subset \partial \nu$ is a simple closed multi-curve. 

If $\pi(C)$ is a simple closed curve, put $b_C := (T_{\pi(C)})^{\frac{k}{e(C)}} \in QP(n)$. 
Then by (\ref{eqn:liftC}) we have $\Psi(b_C)=\Psi((T_{\pi(C)})^{\frac{k}{e(C)}}) = T_{[C]}$.

If $\pi(C)$ is a simple arc, i.e., $k=2$ and $e(C)=1$, put $b_C := h_{\pi(C)} \in QP(n)$ the positive half twist about $\pi(C)$. 
Then we have $\Psi(b_C)= \Psi(h_{\pi(C)}) = T_C = T_{[C]}$. 

For either case, we have 
\[ 
\Psi(b^{d}) = [\phi^{d}] = \Psi \left( \prod_{[C] \in \mathcal{C}} (b_C)^{N([C])} \right). 
\] 
For distinct $[C]$ and $[C']$ the projections $\pi(C)$ and $\pi(C')$ are disjoint because $C_1,\dots,C_m$ are pairwise disjoint. 
Hence, by Proposition~\ref{prop-Birman-Hilden} the braid $b^{2d}$ is a product of positive Dehn twists about pairwise disjoint simple closed curves. 
(We consider $b^{2d}$ rather than $b^d$ so that if $\pi(C)$ is an arc the half twist about $\pi(C)$ becomes  a Dehn twist about a closed curve enclosing $\pi(C)$.) 
Corollary \ref{prop:rootqperidic} shows that $b \in QP(n)$. 
\end{proof}

\begin{proof}[Proof of Claim \ref{claim:subsurface}]

Suppose that $S'$ is an annulus which is a neighborhood of a simple closed geodesic curve $C$ and $[f] = (T_C)^m$. 
By (\ref{eqn:center}) we have $(T_{\psi(C)})^m = (\psi \circ T_C \circ\psi^{-1})^m= \psi \circ (T_C)^m \circ\psi^{-1} = (T_C)^m$, and $\psi(C)$ is isotopic to $C$. 
Hence, $\psi(S'_i)$ is isotopic to $S'_i$. 

Next, suppose that $[f]\in \Mod(S')$ is pseudo-Anosov. 
The property (\ref{eqn:center}) implies that  $f_{i}\psi(S'_i)$ is isotopic to $\psi f_{i}(S'_i) = \psi(S'_i)$. Then either $\psi(S'_i)$ is isotopic to $S'_i$, or by isotopy one can make $\psi(S'_i)$ and $ S'_i$ disjoint. 
The latter possibility cannot be realized by the following reason. 
Take an essential simple closed curve $\alpha \subset S'_i$.
If we can make $\psi(S'_i)$ disjoint from $S'_i$  by isotopy, then  $\psi(\alpha)$ can also be disjoint from $S'_{i}$. 
However, by (\ref{eqn:center})
$\alpha = f_i f_i^{-1}\psi^{-1} \psi(\alpha)=f_i \psi^{-1} f_i^{-1}(\psi(\alpha)) = f_i \psi^{-1}(\psi(\alpha)) = f_{i}(\alpha)$.
This contradicts the assumption that $[f]$ is pseudo-Anosov.
\end{proof}

\begin{proof}[Proof of Claim~\ref{claim:btwist}]

In the case where $S'$ is an annulus,  (\ref{eqn}) is a direct consequence of Claim \ref{claim:subsurface}. 

Assume that $[f]\in\Mod(S')$ is pseudo-Anosov and $S'$ is not an annulus. 
By the symmetry, it is enough to prove the statement (\ref{eqn}) for the case $i=1$. 
Recall that $S'_1\neq S'_p$ and $S'_1\cap S'_p \neq \emptyset$.  
Let $D$ be a connected component of $S'_1 \cap S'_p$.
Since $\phi^{d_0}=id$ on $\partial S'_1 \cup \partial S'_2 \cup \cdots \cup \partial S'_k$ and $\partial D \subset (\partial S'_1 \cup \partial S'_2 \cup \cdots \cup \partial S'_k)$ we have $\phi^{d_{0}}(D)=D$ and $\phi^{d_0}=id$ on $\partial D$.

Since $\partial S'_1$ is geodesic $D$ cannot be a bigon or an annulus. 
Let $\Gamma = \partial (S'_1 \setminus D)$. Then $\Gamma$ is a simple closed multi-curve in $S'_1$. 
Since $S'_1$ is not an annulus 
$\Gamma$ contains an arc $\gamma$ which is essential in $S'_1$. 
Since $\phi^{d_0}=id$ on $\partial D$ the curve $\phi_{1}(\gamma)=\phi^{d_0}(\gamma)$ is isotopic to $\gamma$; hence,  the mapping class $[\phi_{1}] \in\Mod(S')$ cannot be pseudo-Anosov.

On the other hand, $[f]=[f_1|_{S'_1}]\in \Mod(S')$ is pseudo-Anosov, and by (\ref{eqn:center}) $[\phi_1] \in Z([f]) \subset \Mod(S')$. 
Since $[\phi_{1}]$ is not pseudo-Anosov, Lemma~ \ref{lemma:centralizer} shows that $[\phi_1]$ is periodic. 
Namely, there is $d_1 >0$ such that $[\phi_1^{d_1}]$ is a product of Dehn twists about the boundary components of $S'_1$ and we obtain (\ref{eqn}). 

Next we show the second statement of the claim. 
The surface $S'$ is either annular or hyperbolic.  
Suppose that a boundary component $C$ of $S'_1$ is an essential curve in the surface $\nu$. 
Since $C$ is not a boundary component of $\nu$ and $\nu$ is the minimal surface containing $S'_1\cup\cdots\cup S'_k$, there must exist $p \neq 1$ such that $C \cap S'_p \neq \emptyset.$ 

Note that $C$ cannot be a boundary component of $S'_p$ (that is, one side of $C$ is $S'_1$ and the other side of $C$ is $S'_p$)  because in such a case the quotient space $S \slash \iota$ cannot be a topological disk.

If $C$ is not isotopic to a boundary component of $S'_p$ then the descriptions of $[\phi_1^{d_1}]$ and $[\phi_p^{d_p}]$ in (\ref{eqn}) and the definition of $\phi_i$ in (\ref{def of phi_i}) show that $N(C)=0$.

If $C$ transversely intersects a boundary component, say $C'$, of $S'_p$
we also get $N(C)=0$, because otherwise (\ref{def of phi_i}) and (\ref{eqn}) show that $\phi(C')\not\subset S'_p$ which contradicts Claim~\ref{claim:subsurface}.
\end{proof}

We close the paper with an example which shows that Theorem~\ref{thm:general} can be viewed as a generalization of the {\em daisy relation} found in \cite{EMV}.

\begin{example}\label{example-daisy}
Let $F$ be a sphere with $k+1$ holes ($k\geq 3$) that is obtained as the $k$-fold cyclic branched covering $\pi_{F}:F \rightarrow A$ of an annulus $A$ branched at one point. 
Let $a_0,\ldots,a_{k}$ be the boundary components of $F$.
Let $\iota_F:F \rightarrow F$ be a deck transformation defined by  a $2\pi \slash k$ rotation of $F$ about the unique branch point such that $\iota_F(a_{0})=a_{0}$ and $\iota_F(a_i)=a_{i+1}$ for $i=1,\dots,k$. See the left picture of Figure~\ref{fig:daisy}. 
Let $x_i$ ($i=1,\dots,k$) be a simple closed curves on $F$ enclosing $x_0$ and $x_i$ such that $\iota_F(x_{i})= x_{i+1}$.
According to the \emph{daisy relation} as stated in \cite{EMV} we have 
\begin{equation}\label{daisy-1}
T_{x_1}\cdots T_{x_{k}} = T_{a_0}^{k-2}T_{a_1}\cdots T_{a_k} \in \Mod(F). 
\end{equation}
(The case of $k=3$ yields the famous lantern relation.)

Recall the $k$-fold cyclic branched covering $\pi:S\to D$ with $n$ branch points and the deck transformation $\iota:S\to S$. 
Take an embedding $i:F \hookrightarrow S$ such that $i\circ \iota_F(x) = \iota\circ i(x)$ for all $x \in F$.
Let $C = i(x_{1}), A_{0}=i(a_0)$, and $A = i(a_1)$. 
See the right picture in Figure \ref{fig:daisy} for the simplest case $(n,k)=(3,3)$. 
\begin{figure}[htbp]
\begin{center}
\includegraphics*[bb=71 553 413 734, width=130mm]{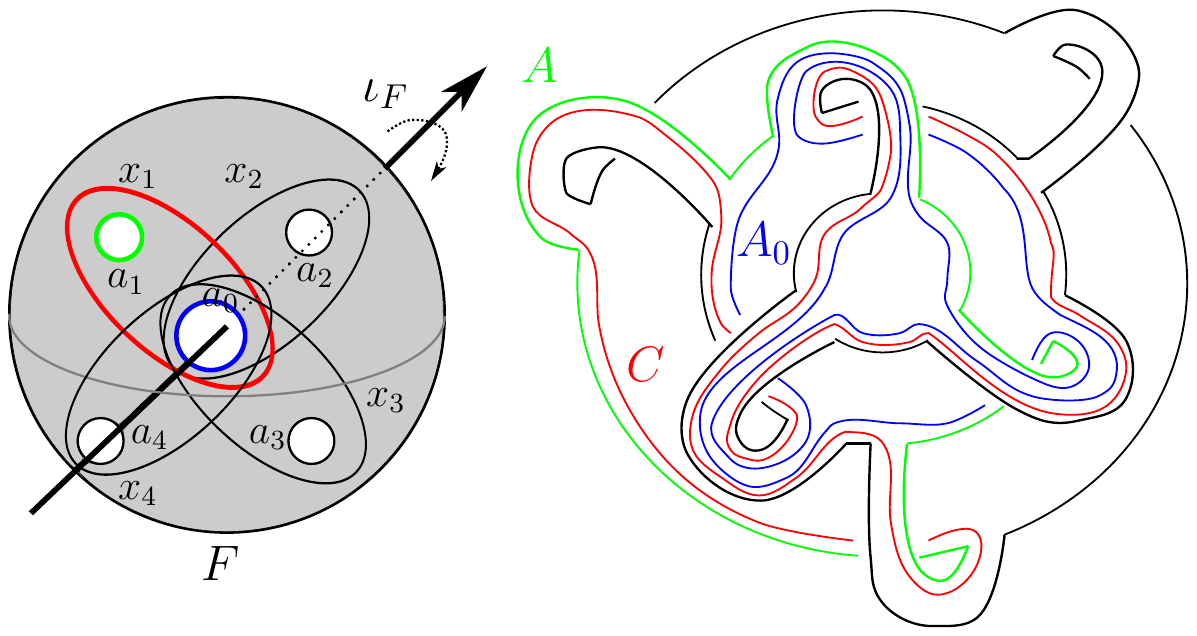}
\end{center}
\caption{Left:  The $(k+1)$ holed sphere $F$ where $k=4$. 
Right: The surface $S$ where $(n,k)=(3,3)$. 
The curves $C, A, A_{0}$ satisfy the daisy (lantern) relation $T_{C}T_{\iota(C)}T_{\iota^{2}(C)} = T_{A_0}T_{A} T_{\iota(A)}T_{\iota^{2}(A)}$.
The surface $S \setminus i(F)$ is the Bennequin surface of a $(2, 3)$ torus knot, which is $A_0$.}
\label{fig:daisy}
\end{figure}

The daisy relation (\ref{daisy-1}) gives 
\begin{equation}\label{eq:daisy}
T_{C}T_{\iota(C)}\cdots T_{\iota^{k-1}(C)} = \left( T_{A_0}^{k-2} \right) \left( T_{A} T_{\iota(A)}\cdots T_{\iota^{k-1}(A)} \right) \in \Mod(S). 
\end{equation}
The left hand side of (\ref{eq:daisy}) is of the form of $[f_1 f_2\cdots f_k]$ which is studied in Theorem~\ref{thm:general}. 
One can show, using \cite[Lemma 3.1]{HKP} and the chain relation \cite[Proposition 4.12]{FM} in mapping class group theory, the term $T_{A_0}^{k-2}$ in the right hand side of (\ref{eq:daisy}) is the image of a quasi-positive braid under the homomorphism $\Psi:B_n\to\Mod(S)$. 
We can also see that the link $A \cup \iota(A)\cup\cdots \cup \iota^{k-1}(A)$ is a $(k, k)$ torus link and the term $T_{A} T_{\iota(A)}\cdots T_{\iota^{k-1}(A)}$ is the image under $\Psi$ of a positive full twist of $k$-stranded braid, which is clearly a quasi-positive element in $B_n$. 
In this sense, Theorem~\ref{thm:general} can be seen as a generalization of the daisy relation. 
\end{example}

\section*{Acknowledgements}
The authors thank Joan Birman, John Etnyre, Amey Kaloti and Rebecca Winarski for useful discussions. 
TI was partially supported by JSPS Grant-in-Aid for Young Scientists (B) 15K17540.
KK was partially supported by NSF grant DMS-1206770 and Simons Foundation Collaboration Grants for Mathematicians 426710.

\end{document}